\documentclass[10pt]{amsart}

\usepackage{latexsym}
\usepackage{amsfonts,amssymb,amsmath,amsthm}
\usepackage{url}
\usepackage{enumerate}

\renewcommand{\epsilon}{\varepsilon}
\newtheorem{thrm}{Theorem}[section]
\newtheorem{lem}[thrm]{Lemma}
\newtheorem{prop}[thrm]{Proposition}

\theoremstyle{definition}

\newtheorem{remark}[thrm]{Remark}
\numberwithin{equation}{section}
\newcommand{\Z}{{\mathbb Z}}
\newcommand{\N}{{\mathbb N}}

\renewcommand{\Pi}{\varPi}

\renewcommand{\epsilon}{\varepsilon}

\newcommand{\R}{{\mathbb R}}

\newcommand{\T}{\mathbb{T}}
\newcommand{\TT}{\mathbb{T}}
\newcommand{\Tn}{\mathbb{T}^n}
\newcommand{\btau}{\tau}
\newcommand{\lj}{{\lambda}_j}

\newcommand{\zed}{{\mathbb{Z}}}

\newcommand{\reals}{{\mathbb{R}}}

\author{Tayeb~A\"issiou}
\address{Department of Mathematics \\
McGill University\\
Montr\'eal, Qc.}
\email{tayeb.aissiou@mail.mcgill.ca}

\date{\today}

\begin{document}
\title[Semiclassical limits of eigenfunctions on flat $n$-dimensional tori]
{Semiclassical limits of eigenfunctions\\
on flat $n$-dimensional tori }

\begin{abstract}
We provide a proof of the conjecture formulated in \cite{Jak97,JNT01} which
states that on a $n$-dimensional flat torus $\T^{n}$, the Fourier transform
of squares of the eigenfunctions $|\varphi_\lambda|^2$ of the Laplacian have
uniform $l^n$ bounds that do not depend on the eigenvalue $\lambda$. The proof
is a generalization of the argument by Jakobson, {\it et al}. for the
lower dimensional cases. These results imply uniform bounds for semiclassical
limits on $\TT^{n+2}$. We also prove a geometric lemma that bounds the number of
codimension-one simplices which satisfy a certain restriction on an
$n$-dimensional sphere $S^n(\lambda)$ of radius $\sqrt{\lambda}$ and use it in
the proof.
\end{abstract}

\maketitle

\section{Introduction}

We let $\Delta$ denote the Laplacian on the $n$-dimensional flat
torus $\T^n=\reals^n/\zed^n$.  The eigenvalues of $-\Delta$ are
denoted by $0=\lambda_0<\lambda_1\leq\lambda_2\leq\ldots$, and the
corresponding eigenfunctions are denoted by $\varphi_j$.  We normalize
$||\varphi_j||_2=1$.

The following Proposition was proved in \cite{Zyg74} for $n=2$, in
\cite{Jak97} for $n=3$, and in \cite{JNT01} for $n=4$.
\begin{prop}\label{prop:small:n}
Let $2\leq n\leq 4$. Then the Fourier series of $|\varphi_j|^2$ have
uniformly bounded $l^n$ norms, where the bound is independent of
$\lambda_j$.
\end{prop}
We remark that it is well-known that the {\em multiplicity} of
$\lambda_j$ becomes unbounded for $n\geq 2$, and therefore so does
$||\varphi_j||_\infty$.

It was conjectured in \cite{Jak97} that the conclusion of
Proposition \ref{prop:small:n} holds for arbitrary $n$.  The main
result of this paper is the proof of that conjecture:

\begin{thrm}\label{thm:main}
For any $n\geq 5$, there exists $C=C_n<\infty$, such that for every
eigenfunction $\Delta\varphi_j+\lambda_j\varphi_j=0,$ $||\varphi_j||_2=1$, the
Fourier series of $g:=|\varphi_j|^2$ satisfies
\begin{eqnarray}\label{ln:main}
\left\| \widehat{g}\right\|_{l^n} \leq
C(n)||\varphi_j||_2^2\nonumber
\end{eqnarray}
\end{thrm}
We stress that the bound $C$ does not depend on the eigenvalue
$\lambda_j$.   The bound $C(n)$ is computed at the end of the
proof and tends to 2 as $n\to\infty$.

Theorem \ref{thm:main} implies (by an argument in \cite{Jak97}) a statement
about limits of eigenfunctions on $\T^{n+2}$. Consider weak limits of the
probability measures $d\mu_j=|\varphi_j|^2dx$, and denote the limit measure as
$\lj\to\infty$ by $d\nu$, one can prove that all such limit measures $d\nu$ are
absolutely continuous in any dimension with respect to the Lebesgue measure on
$\T^n$ (Cf. \cite{Jak97}). Accordingly, by Radon-Nikodym theorem, one can
conclude that $d\nu$ has a density $h(x)\in L^1(\T^n)$ such that
$d\nu=h(x)\,dx$. Then, we consider the Fourier expansion of $h(x)$:
\begin{equation}
\label{eq:defh(x)} h(x)\sim\sum_{\tau\in\Z^n}c_\tau\;e^{i(\tau,x)}
\end{equation}

In dimension $n=2$, it was shown in \cite{Jak97} that the density of
every such limit is a trigonometric polynomial with at most {\it
two} different magnitudes for the frequency. It was also shown
in\cite{Jak97,JNT01} that on $\T^n$ for $3\leq n\leq 6$, the Fourier
expansion of the limit measure $d\nu$ is in $l^{n-2}$, that is,
\begin{equation}\label{smalln:lim:bound}
\sum_{\tau\in\Z^n}|b_\tau|^{n-2}<\infty.
\end{equation}
The proofs in dimensions $4\leq n\leq 6$ used Proposition
\ref{prop:small:n} and results in \cite{Jak97} that reduced
estimates for limits on $\T^{n+2}$ to estimates for eigenfunctions
on $\T^n$. The estimate \eqref{smalln:lim:bound} implies that on
$\T^3$, the density of any limit $d\nu$ has an absolutely convergent
Fourier series, whereas on $\T^4$, we conclude that $h(x)\in
L^2(\T^4)$.

Combining Theorem \ref{thm:main} with the results in \cite{Jak97},
we immediately obtain
\begin{thrm}
\label{thm:limits} Given the Fourier expansion \eqref{eq:defh(x)} of
the limit measure $d\nu$ on $\T^{n+2}$, we have
\begin{equation}
\left(\sum_{\tau\in\Z^{n+2}}|b_\tau|^{n}\right)^{1/n}\leq C(n)<\infty
\end{equation}
\end{thrm}
A generalization of B. Connes' result \cite{Con76} proved in
\cite{Jak97} shows that the constant $C(n)$ appearing in theorem in
theorem \ref{thm:limits} on $\T^{n+2}$ coincides with the constant
in \ref{thm:main} on $\T^n$. The bound $C(n)$ will be computed at
the end of the proof and we will find that it tends to 2 as
$n\to\infty$.

An important question about eigenfunctions of the Laplacian is the following:
given $\varphi(x)$ satisfying $\Delta\varphi_j+\lj\varphi_j=0$ and
$||\varphi||_2=1$ on a general $n$-dimensional smooth Riemannian manifold
$\mathcal{M}$, what is the asymptotic growth rate of the $L^p$ norms of the
eigenfunction? That is, how fast does $||\varphi_j||_{L^p}$ grow as the
eigenvalue $\lj\to\infty$.  

On a two dimensional compact boundaryless Riemannian manifold, Sogge showed in
\cite{Sog88} that for $2\leq p\leq\infty$:
\begin{equation}
\label{eq:Sogge}
||\varphi_{j}||_p\leq C\lambda_j^{\delta(p)}||\varphi_{j}||_2
\end{equation}
where
\begin{equation}
\delta(p) =
\begin{cases}
\frac14\left(\frac12-\frac1p\right), & 2\leq p\leq6 \\
\frac12\left(\frac12-\frac2p\right), & 6\leq p\leq\infty
\end{cases}
\end{equation}
This bound turned out to be sharp on the round sphere $S^2$.

In a remarkable result, Zygmund \cite{Zyg74} provides a uniform bound for the
$L^4$-norm of the eigenfunctions of the Laplacian on $\T^2$. That is,
\begin{equation}
\label{eq:Zyg74}
\frac{||\varphi||_4}{||\varphi||_2}\leq5^{1/4}
\end{equation}
The bound \eqref{eq:Zyg74} provided in \cite{Zyg74} is independent of the
eigenvalue.

Before we mention the next result, we give the following definition:
\begin{equation}
\label{eq:Def}
M_{n,p}(\lambda):=\sup_{(\Delta+\lambda)\varphi=0 \atop
\varphi\text{ on }\T^n}\frac{||\varphi||_p}{||\varphi||_2 }
\end{equation}
The question of the growth rate mentioned earlier can be translated into, what
is the asymptotic behavior of $M_{n,p}(\lambda)$. It is sometimes possible to
obtain uniform bounds (independent of $\lambda$) for $M_{n,p}(\lambda)$ for a
restricted set of eigenvalues. 

In particular, Mockenhaupt proved in \cite{Moc96} the following: given a finite
subset $D=\{q_1,q_2,\ldots,q_k\}$ of prime integers with $q_j\equiv1$
$(\text{mod } 4)$, we consider the set $\Lambda_D$ consisting of all eigenvalues
$\lambda\in\N$ such that all prime divisors $q$ of $\lambda$ with the property
$q\equiv1$ $(\text{mod } 4)$, belong to $D$. Then, for all $\lambda\in\Lambda_D$
and for all $p<\infty$, we have $M_{2,p}\leq C(p,k)<\infty$, where $C(p,k)$ is a
constant.

A legitimate question to ask is whether or not there exists a uniform bound for
$M_{n,p}$ for general $n$ and $p$. The question is still open,
although there exist results about the rate of growth of $M_{n,p}(\lambda)$ as
$\lambda\to\infty$. Bourgain showed in \cite{Bou93} that on $\T^n$ with $n
\geq4$, we have $M_{n,p}\ll\lambda^{(n-2)/4-n/2+\varepsilon}$  for $p\geq
2(n+1)/(n-3)$.


We notice that theorem \ref{thm:main} does not imply a bound on eigenfunctions
since there is no converse to Hausdorff-Young inequality. For $1<p\leq2\leq
q<\infty$ with $p^{-1}+q^{-1}=1$, we have:
\begin{equation}
\label{eq:HYI}
||b_\btau||_{l^q}\ll ||\varphi||_{L^{2p}}^{2}.
\end{equation}

Although the bound $C(n)$ from Theorem \ref{thm:main} does not depend on the
eigenvalue $\lambda$, it does not give us information about the bound $M_{n,p}$
in \eqref{eq:Def}.

In recent papers \cite{BR09,BR10}, J. Bourgain and Z. Rudnick considered upper
and lower bounds for the $L^p$ norms of the the restriction of
eigenfunctions of Laplacian to smooth hypersurfaces of $\TT^n$ with nonvanishing
curvature for $n=2,3$ . They showed that 
\begin{equation}
\nonumber
||\varphi_\lambda||_{L^2(\Sigma)}\asymp||\varphi_\lambda||_2
,
\end{equation}
for all eigenfunctions $\varphi_\lambda$ of the Laplacian on $\TT^n$ with
$\lambda\geq\Lambda$ for some $\Lambda$ that depends only on the hypersurface
$\Sigma$.

There exist bounds for the $L^\infty$ norm of the eigenfunctions as well.
H\"ormander showed (cf. \cite{H-IV83,H2}) that on any compact Riemannian
manifold $M$, we have  
\begin{equation}
\label{Hormander}\nonumber
||\varphi_\lambda||_\infty \leq C\,\lambda^{\frac{n-1}{4}},
\end{equation}
where $n$ is the dimension of the manifold $M$. This bound is attained for some
manifolds such as $S^{n}$, but not for others such as $\TT^n$. Manifolds for
which this bound is sharp are called manifolds with {\it maximal eigenfunction
growth}. 

Y. Safarov studied the asymptotic behavior of the spectral function, the
remainder in Weyl's law, and of eigenfunctions in many papers including
\cite{Saf1,Saf2}.

C. Sogge, J. Toth and S. Zelditch studied, in a series of papers
(Cf. \cite{STZ09,SZ02,TZ02}) the following question: what characterizes
the manifolds with maximal eigenfunction growth?  

They established that the manifolds with maximal eigenfunction growth must have
a point $x$ where the set of geodesic loops at that point has a positive measure
in $S_x^*M$. The converse turned out to be false as they constructed a 
counterexample in \cite{SZ02}.

An older question of the same type is: how fast does the spectral
function and the remainder term in Weyl's formula grow as $\lambda\to\infty$?
The spectral function is given by:
\begin{equation}
\label{eq:specfunc}
N_{x,y}(\lambda)=\sum_{0<\sqrt{\lambda_j}<\sqrt{\lambda}}
\varphi_j(x)\overline{\varphi_j(y)}
\end{equation}
If we consider the diagonal when $x=y$, we obtain $N_{x,x}(\lambda)$. If we
integrate the latter over the volume of the manifold $M$ (assumed to be
compact), we obtain the eigenvalue counting function $N(\lambda)$ defined by:
\begin{equation}
N(\lambda)=\#\lbrace{{\lambda_i}<\lambda\rbrace}.
\end{equation}
The remainder term in Weyl's formula is given by:
\begin{equation}
\label{eq:Weylrem}
R(\lambda)=N(\lambda)-c_n\text{ vol} (M)\,\lambda^{n/2},
\end{equation}
where $c_n$ is a constant that depends on the dimension $n$. 

The asymptotic behavior of the spectral function and the remainder term were
studied by many people, cf. \cite{Av, DG, H2, Lev, Safbook} and the
references therein for a detailed exposition of the subject.

The results of this paper appear in \cite{Ais09}.

\noindent{\bf Acknowledgements:}
I would like to thank professor Dmitry Jakobson for pointing out this problem to me, for the stimulating conversations we had, as well for his moral support. The author was partially supported by FQRNT.

\section{Proof of the main result}
Let us define the notation that will be used throughout the argument. For
$\varphi_j(x)$, an $L^2$-normalized eigenfunction of the Laplacian on an
$n$-dimensional torus $\Tn=\R^n/\Z^n$ with eigenvalue $\lj$, we let its
Fourier expansion be:
\begin{eqnarray}
\varphi_j(x)&\sim&\sum_{\eta\in\Z^n\atop|\eta|^2={\lj}}a_\eta\,e^{i(x,\eta)}
\label{eq:taylorphi}
\nonumber
\end{eqnarray}
The Fourier series of $g(x)=|\varphi_j(x)|^2$ (recall the definition from the
introduction) is as follows:
\begin{eqnarray}
\label{eq:fourexp}
|\varphi_j(x)|^2&\sim&\sum_{\btau=\xi-\eta\atop|\xi|^2=|\eta|^2=\lj}\,b_\btau
\,e^{i(x,\btau)} \label{eq:fouphipsi1}\\
b_\btau&=&\sum_{\xi-\eta=\btau\atop |\xi|^2=|\eta|^2=\lambda_j} \,a_\xi
\bar{a}_\eta
 \\
\label{ddf}
\sum_{\eta\in\Z^n\atop |\eta|^2=\lj}{|a_\eta|^2}&\equiv&1
\label{eq:normalizationphi}
\end{eqnarray}

We will write ${\bf S}^{n-1}(\lj)$ for the $(n-1)$-sphere of radius
$\sqrt{\lj}$ and  $S_{n-1,\lj}$ for the set of lattice points on
${\bf S}^{n-1}(\lj)$.

In the spirit of this new notation, the last three equations may be written
as follows:
\begin{eqnarray}
\label{eq:fourexp'}
|\varphi_j(x)|^2&\sim&\sum_{\btau=\xi-\eta\atop\xi,
\eta\,\in S_{n-1,\lj}}\,b_\btau
\,e^{i(x,\btau)} \label{eq:fouphipsi2}\\
b_\btau&=&\sum_{\xi, \eta\,\in S_{n-1,\lj}\atop\xi-\eta=\btau}
\,a_\xi \bar{a}_\eta \\
\label{ddf2}
\sum_{\eta\,\in S_{n-1,\lj}}{|a_\eta|^2}&\equiv&1 \label{eq:normalizationphi1}
\end{eqnarray}

We can assume, without loss of generality, the coefficients $a_\xi$ to be real
and then we have $|a_\xi|=|\bar a_\xi|=|a_{-\xi}|$. For the case where
$\btau=\bf0$, we have:
\begin{eqnarray}
\label{eq:normzeroterm}
b_{\bf0}=\sum_{0=\btau=\xi-\eta}a_\xi \bar{a}_{\eta}=\sum_{\xi\,
\in S_{n-1,\lj}}|a_\xi|^2=1 .
\end{eqnarray}
The proof of Theorem \ref{thm:main} requires a lemma that will be
proved at the end of this section.
\begin{lem}
\label{lem:geometric}
Given $n$ points $\{\xi_i\}_{i=1}^n$ on ${\bf S}^{n-1}(\lj)\cap\Z^n$, no two of
which are diametrically opposite, that form codimension-one simplex, assume that
there exists $\btau\in\Z^n$ and another $n$ points $\{\eta_i\}_{i=1}^n$ on ${\bf
S}^{n-1}(\lj)\cap\Z^n$ such that 
\begin{equation}
\label{par:translate}
\xi_i-\eta_i=\pm\tau,\qquad \forall 1\leq i\leq n.
\end{equation}
Then, there can be at most $2^{n-1}$ such {\it different} vectors $\btau$
satisfying \eqref{par:translate}.
\end{lem}

\begin{remark}\label{rem:lem}
Given $m>n$ points on ${\bf S}^{n-1}(\lj)\cap\Z^n$, we will still have the same
bound, $2^{n-1}$ on the number of possible $\btau$'s. In other words, adding
more points augments the number of restrictions, which, in principle, might
reduce the number of possibilities for the {\it different} $\btau$'s.
\end{remark}
\begin{remark}\label{rem:evalind}
We also notice that the bound we obtained is independent of the eigenvalue
$\lj$. This fact is crucial in the proof of Theorem \ref{thm:main}.
\end{remark}

The proof of Theorem \ref{thm:main} is done by strong induction, the base case
being done in \cite{Jak97} for the case of $n=3$ and in \cite{JNT01} for the
case of $n=4$. We will provide a proof for the case of $n=5$ first. This will
give a feeling of how the proof of the general case goes. 

\begin{proof}[Proof of theorem \ref{thm:main} for the case $n=5$]
The aim of the following calculations is to bound the sum
$\sum_{\tau}{|b_\tau|^5}$. Given \eqref{eq:normzeroterm}, we will
consider the sum with nonzero $\tau$:
\begin{eqnarray}
\label{eq:interchange}
\sum_{\tau\not=\bf 0}{|b_\tau|^5}&\leq&\sum_{\tau\not=\bf
0}\left(\sum_{\xi_j-\eta_j=\tau}\prod_{j=1}^5|a_{\xi_j}||a_{\eta_j}|\right)
\end{eqnarray}
The trick that we shall use is to bound the right-hand side of
\eqref{eq:interchange} by:
\begin{eqnarray}
\label{eq:boundedwith}
\sum_{\tau\not=0}{\sum_{\xi_i-\eta_i=\tau}{\frac12\left(\prod_{i=1}^5|a_{\xi_i}
|^2+\prod_ { i=1 } ^5{|a_{\eta_i}|}\right)}}
\end{eqnarray}
then, we interchange the order of summation in \eqref{eq:boundedwith} and
finally we use lemma \ref{lem:geometric} to obtain a finite upper bound.

In doing so, we will encounter several configurations of the points $\xi_i$'s on
${\bf S}^4(\lj)\cap\Z^5$. Each configurations needs to be studied separately. An
obvious case is when two or more points $\xi_i$ coincide, equation
\eqref{eq:interchange} reduces to,
\begin{eqnarray}
\label{eq:twocoincide}
\sum_{\btau\not=0}\sum_{\xi_0-\eta_0=\btau}|a_{\xi_0}|^2|a_{\eta_0}
|^2\left(\sum_{\xi_i-\eta_i=\xi_0-\eta_0}\left(\prod_{i=3}^5|a_{\xi_i}||a_{
\eta_i}|\right)\right)
\end{eqnarray}
and one can bound the terms $|a_{\xi_i}||a_{\eta_i}|$ inside the product of
\eqref{eq:twocoincide} by $\frac12(|a_{\xi_i}|^2+|a_{\eta_i}|^2)$. Then, we can
bound this case by,
\begin{equation}
\label{eq:n5otherbound}
\frac1{2^3}\;\sum_{\btau\not=0}\sum_{\xi_0-\eta_0=\btau}|a_{\xi_0}|^2|a_{
\eta_0}|^2 \left(\sum_{\xi,\eta\in S_{4,\lj}}|a_\xi|^2|a_{\eta}|^2\right)
\end{equation}
where the former is bounded by $\frac{1}{2^3}$.

Now, we may suppose that no two points coincide. We end up with five points in
$\R^5$. These points will either lie in a 4 dimensional affine subspace (where
they will form a 4-simplex), a 3 dimensional affine subspace or a 2 dimensional
affine subspace.

In the case where the points form a 4-simplex, we can use lemma
\ref{lem:geometric} and interchange the order of summation in
\eqref{eq:boundedwith} as follows,
\begin{equation}
\label{eq:4simplexbound}
\frac12\sum_{\xi_i\in S_{4,\lj}}\sum_{\btau\not=0}\sum_{\xi_i-\eta_i=\pm\btau}
{\left(\prod_{i=1}^5|a_{\xi_i}|^2+\prod_{i=1}^5|a_{\eta_i}|^2\right)}.
\end{equation}
The former will be bounded by
\begin{equation}
\label{eq:4simplexveryfinal}
\frac12\sum_{\xi_i\in S_{4,\lj}}2^4\cdot2\prod_{i=1}^5|a_{\xi_i}|^2
\end{equation}
which by the $L^2$ normalization will not exceed $2^4$.

In the case where the points $\xi_i$ lie in a 3 dimensional affine subspace
namely $\alpha$, they will form a codimension 2 simplex. There will be 3
different configurations that need to be considered.

The first case is when $\lbrace{\xi_i\rbrace}_{i={1\ldots5}}\in\alpha$ and at
least one of the $-\eta_i\not\in\alpha$. Without loss of generality, we may
suppose that $-\eta_5\not\in\alpha$. Then, the simplex formed by
$(\xi_1,\xi_2,\xi_3,\xi_4,-\eta_5)$ is a parallel translate of the simplex
formed by $(\eta_1,\eta_2,\eta_3,\eta_4,-\xi_5)$ and these simplices do {\it
not} lie in a 3-dimensional subspace. They form a non-degenerate 4-simplex.
Hence, we are reduced to the case just studied above and we obtain the same
bound, that is, $2^4$.

In the next case, we suppose that the points
$\lbrace{\xi_i\rbrace}\in\alpha$,
$\lbrace{-\eta_i\rbrace}\in\alpha$ but
$\lbrace{\eta_i\rbrace}\not\in\alpha$ for all $i=1\ldots5$. The trick we will
be using is to consider the subspace that contains both $\alpha$ and $\eta_1$
say, namely $\gamma$. The subspace $\gamma$ is a 4 dimensional subspace that
contains $\bf 0$ since both $\eta_1$ and $-\eta_1$ lie in $\gamma$. Thus,
$\gamma\cap {\bf S}^4(\lj)$ is the great 3-sphere, where the great $k$-sphere
is defined to be the intersection of ${\bf S}^n(\lj)$ with a $k $ dimensional
hyperplane passing through the origin. Hence, by lemma \ref{lem:geometric} and
remark \ref{rem:lem}, we have the same bound on the number of $\btau$'s as to
have 4 points on $S_{3,\lj}$, and this will lead to a bound of $2^3$.

The last scenario that needs to be considered in the case where
$\lbrace{\xi_i\rbrace}_{i=1\ldots5}\in\alpha$ is when
$\lbrace{-\eta_i\rbrace}_{i=1\ldots5}\in\alpha$ and at least one of the
$\eta_i\in\alpha$, say $\eta_1$. Since both $\eta_1$ and $-\eta_1$ are in
$\alpha$, ${\bf 0}\in\alpha$ and all of $\pm\eta_i\,\,,\pm\xi_i\in\alpha$.
Hence, $\alpha\cap {\bf S}^4(\lj)$ is the great 2-sphere. Once again, lemma
\ref{lem:geometric} and remark \ref{rem:lem} will lead us to a bound that is
equal to $2^2$.

It may happen that the points lie in a 2-dimensional affine subspace say,
$\beta$. We will study the possible cases in the same manner we did previously.
In the first case, we suppose that $\{\xi_i\}_{i=1\ldots5}\in\beta$ with
$\{-\eta_i\}\in\beta$ for all $i$. We consider the 3-dimensional subspace
$\gamma_1$ that contains both $\beta$ and $\eta_1$ say. Then,
${\bf0}\in\gamma_1$, which implies that $\pm\eta_i,\pm\xi_i$ all lie in
$\gamma_1\cap {\bf S}^4(\lj)$, which is the great 2-sphere. We are back in
one of the cases studied previously and once again, lemma \ref{lem:geometric}
and remark \ref{rem:lem} will guarantee us a bound of $2^2$.

In the very last case, we lose a bit of control on where the $\eta_i$ might be.
We let $\xi_i\in\beta$, but at least one of the $-\eta_i\not\in\beta$, $-\eta_5$
say Then, the points $\{\xi_1,\xi_2,\xi_3,\xi_4,\eta_5\}$ lie in a 3-dimensional
affine subspace and we are back to case where the $\xi_i\in\alpha$. Hence, we
have a total bound equal to $2^2+2^3+2^4=28$.

Summing all the bounds, we obtain $C(n)\approx2.384729...$
\end{proof}

\begin{proof}[Proof of the general case]
We shall now turn into the proof of the general case, that is, the sum
\eqref{eq:nSum} given below is convergent for any $n$. The proof is done by
strong induction. That is, we suppose that the sum \eqref{eq:nSum} is bounded
in any dimension $k<n$.
\begin{eqnarray}\label{eq:nSum}
\sum_{\btau\in\Z^n\cap {\bf S}^{n-1}(\lj)}|b_\btau|^n&=&
1+\sum_{0\not=\btau\in\Z^n\cap {\bf S}^{n-1}(\lj)}|b_\btau|^n
\end{eqnarray}
As in the proof of the $n=5$ case, we have,
\begin{equation}
\label{eq:nproductnterms}
\sum_{\btau\not=\bf0}|b_\btau|^n\leq\sum_{\btau\not=\bf0}\sum_{
\xi_i-\eta_i=\btau }\prod_{i=1}^n|a_{ \xi_i } ||a_ { \eta_i } |
\end{equation}
The same trick is used as before, that is, we will bound the right-hand side
of \eqref{eq:nproductnterms} by \eqref{eq:nestimatenterms}, then interchange
the order of summation in the latter, and finally use Lemma \ref{lem:geometric}
to obtain a finite upper bound.

\begin{equation}
\label{eq:nestimatenterms}
\sum_{\btau\not=\bf0}\sum_{{\xi_i}-\eta_i=\btau}\frac{1}{2}\left(\prod_{i=1}
^n|a_{\xi_i}|^2+\prod_{
i=1}^n|a_{\eta_i}|^2\right)
\end{equation}

Once again, several cases need to be studied. We will do so in the same manner
as for the $n=5$ case. Instead of 5 points, we now have $n$ points
$\lbrace\xi_i\rbrace_{i=1}^n$ on the surface of the sphere ${\bf S^{n-1}}(\lj)$

The trivial case where two or more points coincide gives a bounded contribution
to the sum \eqref{eq:nSum} that is equal to $\frac1{2^{n-2}}$ by the same
computations done in the $n=5$ case. In the subsequent cases, we may assume that
no two points $\xi_i$ coincide. 

The second trivial case is when the points $\lbrace\xi_i\rbrace$ form a
non-degenerate codimension-1 simplex. A change of order of summation in
\eqref{eq:nestimatenterms} and Lemma \ref{lem:geometric} yield a bound
equal to $2^{n-1}$. 

The non trivial cases are when the points $\lbrace\xi_i\rbrace$ lie in smaller
subspaces. Providing an upper bound to each of these cases finishes the proof.

The first of such non trivial cases is when the points $\lbrace\xi_i\rbrace$ lie
in a $(n-2)$ dimensional affine subspace, namely $\alpha_{n-2}$. Let us suppose
$\{\xi_i\}_{i=1}^n\in\alpha_{n-2}$ with all the
$\{-\eta_i\}\in\alpha_{n-2}$ as well. If either one of the $\eta_i$'s or
$-\xi_i$'s is an element of $\alpha_{n-2}$, then the origin ${\bf
0}\in\alpha_{n-2}$, which implies that $\alpha_{n-2}\cap{\bf S}^{n-1}(\lj)$ is
the great $(n-2)$-sphere. Hence, we have $n$ points on $S_{n-2,\lj}$ and by the
induction hypothesis, this gives us a bounded contribution to the sum 
\eqref{eq:nSum}. Suppose now that none of the $\eta_i$'s or $-\xi_i$'s is an
element of $\alpha_{n-2}$. Then, we consider the subspace $\beta_{n-2}$
containing both $\alpha_{n-2}$ and $\eta_1$ say. We get an $(n-1)$-dimensional
subspace including $\bf 0$, and $\beta_{n-2}\cap {\bf S}^{n-1}(\lj)$ is the
great $(n-2)$-sphere. Remark \ref{rem:lem} implies that the resulting case is
one of the cases in our induction hypothesis and this gives a bounded
contribution to the sum  \eqref{eq:nSum}.

In order to prove it for the rest of the cases; i.e., when the points
$\{\xi_i\}$ lie in a $(n-k)<(n-2)$ dimensional affine subspace, namely
$\alpha_{n-k}$, we will use a {\it second} (reversed) induction on the dimension
of the affine subspace $\alpha_{n-k}$ where the points $\{\xi_i\}$ might lie.
That is, assuming we have a bounded contribution from all the $\alpha_{n-k+1}$ 
for some $k$ with $3<k<(n-1)$, we will prove that we have a bounded contribution
from the case where the $\{\xi_i\}\in\alpha_{n-k}$. Once again, we have the two
subcases depending on whether or not $-\eta_j$ belong to $\alpha_{n-k}$. 

For the first subcase, we may assume, without loss of generality,  that
$-\eta_1\not\in\alpha_{n-k}$. Then, the simplex $(-\eta_1,\xi_2,\ldots,\xi_n)$
is a parallel translate of $(-\xi_1,\eta_2,\ldots,\eta_n)$ and the last two
simplices lie in a $(n-k+1)$-dimensional subspace. Hence, we are reduced to the
{\it second} induction hypothesis which yields a bounded contribution to the
sum \eqref{eq:nSum}.

Let us now turn our attention to the second subcase: if all the
$\{\xi_i\}_{i=1}^n$ and $\{-\eta_i\}_{i=1}^n$ lie in $\alpha_{n-k}$ with {\it
none} of the $\eta_i$'s in $\alpha_{n-k}$, we consider the subspace
$\beta_{n-k}$ containing both $\alpha_{n-k}$ and $\eta_1$ say. This is a
$(n-k+1)$-dimensional subspace that includes $\bf 0$. We can see that 
$\beta_{n-k}\cap {\bf S}^{n-k+1}(\lj)$ is the great $(n-k)$-sphere. Hence, we
have $n$ points on $S_{n-k,\lj}$ and by the strong {\it first} induction
hypothesis, we obtain a finite contribution from this subcase to the sum
\eqref{eq:nSum}.

We note that if all the $\{\xi_i\}_{i=1}^n$ and $\{-\eta_i\}_{i=1}^n$ lie in
$\alpha_{n-k}$ with { at least one} of the $\eta_i$'s in $\alpha_{n-k}$, then 
${\bf 0}\in\alpha_{n-k}$ and $\alpha_{n-k}\cap {\bf S}^{n-1}(\lj)$ is the great
$(n-k-1)$-sphere and this case gives a bounded contribution to the sum
\eqref{eq:nSum} by once again the strong {\it first} induction hypothesis.

We have exhausted all the possible cases, each giving a bounded contribution to
the sum \eqref{eq:nSum}. Therefore, the sum is bounded and this finishes the
proof of the conjecture in \cite{Jak97}.
\end{proof}

We may now provide a proof for the geometric Lemma \ref{lem:geometric}. 

\begin{proof}[Proof of lemma \ref{lem:geometric}]
Suppose we are given $\{\xi_i\}_{i=1}^n$, $n$ points on $S_{n-1,\lj}$, no two of
which are diametrically opposite, and such that the simplex with vertices
$\{\xi_i\}_{i=1}^n$ is non-degenerate. That is, the points $\{\xi_i\}_{i=1}^n$
cannot be in any (affine) subspace of dimension {\it strictly} less than $n-1$.
Then, given $n$ equal parallel ``chords" $\{{\bf v_i}\}_{i=1}^n$ of
$S_{n-1,\lj}$ (not equal to $\overline{\xi_i\xi_j}, \forall i,j$) such that
$\xi_i$ is an endpoint of $\bf v_i$, we denote the other endpoint of ${\bf v_i}$
by $\eta_i$ and the diametrically opposite points of $\xi_i$ (respectively
$\eta_i$) by $\xi_i'$ (respectively $\eta_i'$). The question we would like to
pose is: {\it where on $S_{n-1,\lj}$ can $\{\eta_i\}_{i=1}^n$ lie?} We will see
that there are {\it finitely} many places where the $\{\eta_i\}_{i=1}^n$ can be.
In fact, there are $\left[\frac{n}{2}\right]$ different scenarios, and we will
study each of them. 

If $\overline{\xi_i\eta_i}$ are equal $\forall i$, then $\eta_1=\eta_i+
\overline{\xi_i\xi_1}$ for all $i=1\ldots n$. Hence, the points
$\eta_1+\overline{\xi_1\xi_i}$ lie on $S_{n-1,\lj}$ $\forall i$. Since ${\bf
S}^{n-1}(\lj)$ is strictly convex, there is at most {\it one} point (other than
$\xi_1$), namely $\eta_1$, for which the points $\eta_1+\overline{\xi_1\xi_i}$
for all $i=1\ldots n$ lie on $S_{n-1,\lj}$.

In the next scenario, we suppose $\overline{\xi_i\eta_i}$ are equal for all $i$,
except at one point $k$, where $\overline{\xi_i\eta_i}= \overline{\eta_k\xi_k}$.
Then, the points $\eta_1+ \overline{\xi_1\xi_i}$ for all $i\not=k$ and $\eta_1+
\overline{\xi_1\xi_k'}$ lie on $S_{n-1,\lj}$. Again, by the convexity of ${\bf
S}^{n-1}(\lj)$ and the fact that $\{\xi_i\}$ form a codimension-1 simplex, there
is at  most one point (other than $\xi_1$), namely $\eta_1$, for which the
points $\eta_1+\overline{\xi_1\xi_i}$ for $i\not=k$ and $\eta_1+
\overline{\xi_1\xi_k'}$ lie on $S_{n-1,\lj}$. However, the last equation gives
us at most $one$ possibility for $\eta_1$ for every $k=1\ldots n$. Hence, we
have a total of $n= {n\choose 1}$ possibilities for $\eta_1$.

In the next case, we assume $\overline{\xi_i\eta_i}$ are equal for all
$i\not=k,l$, where $\overline{\xi_i\eta_i}=\overline{\eta_k\xi_k}=
\overline{\eta_l\xi_l}$. Here again, $\eta_1=\eta_i+\overline{\xi_i\xi_1}$
for all $i\not=k,l$ and $\eta_1=\eta_k'+\overline{\xi_k'\xi_1}=\eta_l'+
\overline{\xi_l'\xi_1}$, making the points $\eta_1+\overline{\xi_1\xi_i}$
for $i\not=k,l$, $\eta_1+\overline{\xi_1\xi_k'}$ and $\eta_1+
\overline{\xi_1\xi_l'}$
lie on $S_{n-1,\lj}$. The convexity of ${\bf S}^{n-1}(\lj)$ implies the
uniqueness of such $\eta_1\not=\xi_1$ for every pair $k,l$. Hence, we have $n
\choose 2$ possibilities for $\eta_1$ in this scenario.

Similarly, we will get $n \choose 3$ for the next and so on, until $n \choose
n$. However, the $n \choose n$ case is the same as the very first case $n\choose
0$ in which we will simply change the sign of all the vectors
$\overline{\xi_i\eta_i}$. The $({n-1})^{\text{th}}$ scenario is similar to the
second scenario, and so on; hence, counting twice every case. The total number
of possibilities will be the sum of the possibilities in every scenario and is:
\begin{eqnarray}
\frac12\sum_{k=0}^n {n\choose k}=2^{n-1}
\end{eqnarray}
\end{proof}

The bound follows from the proof of Theorem \ref{thm:main}, and use the bounds
given by Lemma~\ref{lem:geometric}. We do not claim that $C(n)$ is a sharp bound. In fact we suspect that one can improve the bound obtained in Lemma~\ref{lem:geometric} and get a better final bound that would approach 1 in the limit. In our setting, and for $n\geq5$ the result will be:
\begin{equation}
\label{eq:finalbound}
C(n)=\left(2^{2-n}+\left(\frac{5n}4-4\right)2^n+5\right)^{1/n}
\end{equation}
It is clear that $C(n)\to2$ as $n\to\infty$.

\end{document}